\documentclass[a4paper,12pt]{article}
\usepackage[a4 paper, margin=2.5cm]{geometry}

\usepackage[normalem]{ulem}

\usepackage{hyperref}

\usepackage{amsmath,amsfonts,amsthm,amssymb}
\usepackage{color,enumitem,graphicx}

\usepackage[T1]{fontenc}
\usepackage{lmodern}

\numberwithin{equation}{section}

\newtheorem{theorem}{Theorem}[section]
\newtheorem{lemma}[theorem]{Lemma}

\theoremstyle{definition}

\theoremstyle{remark}
\newtheorem{remark}[theorem]{Remark}

\newtheorem{q}[theorem]{Question}

\newcommand{\T}{\mathrm{T}}
\newcommand{\Id}{\mathrm{Id}}
\newcommand{\A}{\mathrm{A}}
\newcommand{\ed}{\mathrm{d}}
\newcommand{\de}{\partial}

\newcommand{\End}{\mathrm{End}}
\newcommand{\SO}{\mathrm{SO}}
\newcommand{\Or}{\mathrm{O}}
\newcommand{\V}{\mathcal{V}}
\newcommand{\so}{\mathfrak{so}}

\newcommand{\R}{\mathbb{R}}
\newcommand{\N}{\mathbb{N}}

\newcommand{\la}{\langle}
\newcommand{\ra}{\rangle}

\newcommand{\h}{\mathcal{H}}

\newcommand{\emb}{\hookrightarrow}

\newcommand{\In}{\subset}
\newcommand{\Om}{\Omega}
\newcommand{\om}{\omega}
\newcommand{\dl}{{\delta}}
\newcommand{\Dl}{{\Delta}}

\newcommand{\al}{{\alpha}}

\newcommand{\id}{\,\,\ed}

\newcommand{\D}{{\nabla}}

\newcommand{\ti}[1]{{\tilde{#1}}}
\newcommand{\eps}{{\varepsilon}}
\newcommand{\fr}[2]{\frac{#1}{#2}}

\newcommand{\db}{\overline{\partial}}
\newcommand{\nn}{\nonumber}

\usepackage{todonotes}

\newcommand{\cN}{{\mathcal N}}

\newcommand{\cR}{{\mathcal R}}

\newcommand{\cV}{{\mathcal V}}

\begin{document}

\title{A Note on Moving Frames along Sobolev Maps and the Regularity of Weakly Harmonic Maps}
\author{Luigi Appolloni and Ben Sharp}
\maketitle
\begin{abstract}
 The purpose of this note is twofold. First we show that, for weakly differentiable maps between Riemannian manifolds of any dimension, a smallness condition on a Morrey-norm of the gradient is sufficient to guarantee that the pulled-back tangent bundle is trivialised by a finite-energy frame over simply connected regions in the domain. This is achieved via new structure equations for a connection introduced by Rivi\`ere in the study of weakly harmonic maps, combined with Coulomb-frame methods and the Hardy-BMO duality of Fefferman-Stein. 
 
 We also prove that for weakly harmonic maps from domains of any dimension into closed homogeneous targets, a smallness condition on the BMO seminorm of the map is sufficient to obtain full regularity. 
\end{abstract}
\section{Introduction}

Since the work of Wente \cite{wente}, so-called div-curl structures or ``Wente terms'' have been known to appear crucially in the study of geometric PDE and calculus of variations in the large. Such terms are generally of the form $E\cdot D$ for vector fields in conjugate Lebesgue spaces $E,D\in L^p, L^{p^\ast}$ for which one is divergence-free whilst the other is curl-free: \emph{a-priori} only in $L^1$, these terms lie in the slightly smaller Hardy space $\h^1$ which enjoys better properties in relation to weak convergence and under convolution by singular integrals (and hence elliptic regularity theory) see Coiffman-Lions-Mayers-Semmes \cite{CLMS93} for generalisations and links to classical harmonic analysis. In his celebrated work on the regularity theory of harmonic maps from surfaces, Fr\'ed\'eric H\'elein \cite{helein_conservation} used a moving frame technique along the map to re-write the equations in such a way that Wente-type terms naturally appear. This was extended by Fabrice Bethuel \cite{Be93} for weakly harmonic maps $u:B_1^m\to\cN$ from higher dimensional domains,  under a smallness condition on the $M^{2,m-2}$ Morrey norm of the gradient (which is just the $L^2$-norm when $m=2$):
\[
\|\ed u\|_{M^{2,m-2}(B_1)}:=\sup_{x \in B_1^m, r > 0} 
 \left( r^{2-m}\int_{B_r(x) \cap B_1^m}  |\ed u|^2   \, dx \right)^{1/2}.
\]   The appearance of Wente terms is sufficient to conclude that a weakly harmonic map whose gradient is small in this Morrey-norm, is in fact smooth. This smallness condition is guaranteed for weakly stationary harmonic maps away from a closed set $S\In B_1$ satisfying $\h^{m-2}(S)=0$. Here are throughout, $\cN^n\emb \R^d$ is an $n$-dimensional submanifold of $\R^d$ equipped with the induced metric and $B_1^m\In \R^m$ denotes the open unit ball.

In \cite{riviere_inventiones} Tristan Rivi\`ere established a deep and remarkable link between the above regularity theory and systems of the form
\begin{equation}\label{eq:omeq}
	-\Dl u^i = \Om^i_j \cdot \ed u^j \qquad \text{for $u\in W^{1,2}(B_1^m,\R^d)$} \qquad \text{and $\Om\in L^2(B_1^m,\so(d)\otimes \T^\ast \R^m)$.}
\end{equation}
Rivi\`ere found a new gauge in order to re-write the above as a conservation law under an appropriate smallness-condition on $\Om$,  when $m=2$. In particular he uncovered hidden Wente terms and showed that solutions were continuous. This was extended by Rivi\`ere-Struwe \cite{RS08} to higher-dimensional domains, wherein they recover regularity of such systems under the natural assumption that $\Om, \ed u\in M^{2,m-2}$ and $\|\Om \|_{M^{2,m-2}}$ is small. This powerful observation concerning systems of the form \eqref{eq:omeq} has had deep and fruitful consequences in many nearby regularity problems in geometric PDE both in higher-order and non-local settings, see e.g. \cite{riviere_inventiones, dLR11, LR08, S08} and citations thereof.

In particular given a $C^2$-Riemannian manifold $\cN\emb \R^d$ and 
\begin{equation*}
	u\in W^{1,2}(B_1^m,\cN):=\{v\in W^{1,2}(B_1^m,\R^d): v(x)\in\cN \quad \text{for a.e. $x\in B_1^m$}\}	
\end{equation*} 
Rivi\`ere introduced the following specific $\Om$ which we henceforth denote with the lower case $\om = \om(\cN,u): B_1^m \to \so(d)\otimes \T^\ast \R^m$: let $\A\in \Gamma(\T^\ast\cN \otimes \T^\ast\cN \otimes \V\cN)$ be the second fundamental form defined by $\A(X,Y) = (D_X Y)^\bot$ which one extends in the obvious way as a section of $\T^\ast \R^d\otimes \T^\ast \R^d\otimes \T \R^d$ over $\cN$ via $\A(X,Y) :=\A(X^\top, Y^\top)$ and can be written in ambient $\R^d$-coordinates $\{z\}$ via $\A=\A^i_{jk}\ed z^j\otimes \ed z^k\otimes \de_{z^i}$. For $u$ as above define 
\begin{equation}\label{eq:omdef}
	\om = \om^i_j=-(\A^i_{jk}(u) - \A^j_{ik}(u))\ed u^k.
\end{equation}
In particular given a vector $v\in \R^d$ we may see that $\om$ may be equivalently defined without coordinates via
\begin{equation}\label{eq:omv}
	\om v =\om^i_j v^j = -\A(v,\ed u) + \la \A(\ed u, \cdot)^\sharp , v\ra  \in \R^d\otimes \T^\ast \R^m. 
\end{equation}
Rivi\`ere's interest in such an $\om$ is that weakly harmonic maps $u$ solve \eqref{eq:omeq} for this specific $\om$, ($\Dl u = \om\cdot \ed u = \A(\ed u, \ed u)$) and thus the regularity theory for harmonic maps may be reduced to the study of systems in the form \eqref{eq:omeq}. In this article we show that $\om$ as defined in \eqref{eq:omdef} is significant even for arbitrary (i.e. not necessarily harmonic) $u$.

\subsection{Moving frames along Sobolev maps}
The first half of this article concerns \emph{arbitrary} maps $u\in W^{1,2}(B_1^m,\cN)$ and the related $\om$ defined in \eqref{eq:omdef}. We equip $u^\ast (\T \R^d)=B_1^m \times \R^d$ with the connection $\D^\om = \ed + \om$. We uncover some fundamental structure equations for this induced connection which, when combined with Coulomb-frame methods, contain Wente terms. Utilising the $\h^1-BMO$ duality of Fefferman-Stein \cite{FSt72} we have the following:    

\begin{theorem} \label{main}
Let $\cN\emb \R^d$ be $C^2$-Riemannian manifold, $u\in W^{1,2}(B_1^m,\cN)$ and $\om$ be defined by \eqref{eq:omdef}. There exist $\eps = \eps (m,d)>0$, $C=C(m,d)<\infty$ so that if
\[
\|\om\|_{M^{2,m-2}(B_1)}:=\sup_{x \in B_1^m, r > 0} 
 \left(r^{2-m} \int_{B_r(x) \cap B_1^m}  |\om|^2   \, dx \right)^{1/2} < \varepsilon,
\]
then $u^\ast \T\cN$ and $u^\ast \V\cN$ are both trivial in the sense that there exist $$\{e_i\}_{i=1}^n, \{\nu_j\}_{j=n+1}^d \In W^{1,2}(B_1,\R^d)$$ so that $\{e_i(x)\}$ resp. $\{\nu_j(x)\}$ form orthonormal bases of $\T_{u(x)}\cN$ resp. $\V_{u(x)}\cN$ for a.e. $x\in B_1$. Furthermore, $\{e_i\}$ resp. $\{\nu_j\}$ are finite-energy Coulomb frames in the sense that for all relevant $i,j$ we have
$$\ed^\ast(e_i\cdot\ed e_j)=0, \quad \ed^\ast(\nu_i\cdot\ed \nu_j)=0 $$
and $$\max_{i,j}\{\|\ed e_i\|_{M^{2,m-2}(B_1)},\|\ed \nu_j\|_{M^{2,m-2}(B_1)}\}\leq C\|\om\|_{M^{2,m-2}(B_1)}.$$
\end{theorem}

\begin{remark}
	We may conclude the same result, assuming $\|\A\|_{L^\infty(N)}$ is finite (e.g. if $\cN$ is closed), under a smallness condition on $\|\ed u\|_{M^{2,m-2}}$ except both $\eps$ and $C$ would depend additionally on $\|\A\|_{L^\infty(N)}$.  We note that it is not true that a small bound on $\|\ed u\|_{M^{2,m-2}}$ implies that the image of $u$ is contained in a small region e.g. when $m=2$ let $\gamma:[0,\infty)\to \cN$ be any Lipschitz curve with $|\gamma^\prime|=1$ almost everywhere, and consider $\gamma\circ u:B_1\to \cN$ for $u(x) = \eps\log\log(e|x|^{-1})$. However the result above implies that a certain amount of $\|\ed u\|_{M^{2,m-2}}$ is required for the image to ``wrap around'' enough for $u^\ast \T\cN$ to be non-trivial e.g. $u:B_1^3\to S^2$ defined by $u(x) = \frac{x}{|x|}$ satisfies $r^{-1}\|\ed u\|^2_{L^2(B_r(0))}=8\pi$ for all $r>0$ and of course the pulled-back tangent bundle is not trivial in this case.  
\end{remark}

\begin{remark}\label{rmk:h}
We will see below that $\om = \fr12 \cR^{-1}\ed \cR$ where $\cR:B_1\to \Or(d)$ can be interpreted as $G\circ u$ and $G:\cN\to G(n,d)\emb \Or(d)$ is the Gauss map of $\cN$ (see Remark \ref{rmk:totgeo}). Thus the theorem above can be compared directly with Lemma 5.1.4 in \cite{helein_conservation} and be thought of as an extension of this to higher dimensions - see Theorem \ref{thm:gen} below. 
\end{remark}
In light of the above Remarks we pose the following 

\begin{q}
Does the above Theorem remain true if instead of imposing that  $\|\om\|_{M^{2,m-2}}$ or $\|\ed u\|_{M^{2,m-2}}$ is small, we instead require 
$$[\cR]_{BMO(B_1)}:=\sup_{B_r(x)\In B_1} \left(r^{-m}\int_{B_r(x)} |\cR - \overline\cR_{r,x}|^2 \right)^\fr12<\eps, \qquad \overline\cR_{r,x}:=\fr{1}{|B_r(x)|}\int_{B_r(x)} \cR$$
or indeed $[u]_{BMO(B_1)}<\eps$. Of course the best one could hope for is that we end up with an equivalent bound on $[e_i]_{BMO(B_1)}$ and $[\nu_j]_{BMO(B_1)}$ in terms of $[\cR]_{BMO(B_1)}$. 
\end{q}

Theorem \ref{main} is a special case of the following 
\begin{theorem}\label{thm:gen}
    Suppose that $\Pi\in W^{1,2}(B_1^m,\mathfrak{gl}(d))$ is a projection (i.e. $\Pi^2=\Pi$ and $\mathrm{rank}(\Pi)=n$ a.e.), equivalently suppose that $\Pi\in W^{1,2}(B_1,G(n,d))$. Define $\om_\Pi:= \Pi \ed \Pi - \ed \Pi \Pi$.  There exist $\eps = \eps (m,d)>0$, $C=C(m,d)<\infty$ so that if $
\|\om_\Pi\|_{M^{2,m-2}(B_1)} < \varepsilon$
then there exist $$\{e_i\}_{i=1}^n, \{\nu_j\}_{j=n+1}^d \In W^{1,2}(B_1,\R^d)$$ so that $\{e_i(x)\}$ resp. $\{\nu_j(x)\}$ form orthonormal bases of $\Pi(x)\R^d$ resp. $\Pi^\bot(x)\R^d$ for a.e. $x\in B_1$. Furthermore, $\{e_i\}$ resp. $\{\nu_j\}$ are finite-energy Coulomb frames in the sense that for all relevant $i,j$ we have
$$\ed^\ast(e_i\cdot\ed e_j)=0, \quad \ed^\ast(\nu_i\cdot\ed \nu_j)=0 $$
and $$\max_{i,j}\{\|\ed e_i\|_{M^{2,m-2}(B_1)},\|\ed \nu_j\|_{M^{2,m-2}(B_1)}\}\leq C\|\om\|_{M^{2,m-2}(B_1)}.$$ 
In the above $\Pi^\bot(x):=\Id - \Pi$ is the projection onto the orthogonal complement. 
\end{theorem}
\begin{remark}
    As mentioned in Remark \ref{rmk:h}, when $m=2$ this result recovers \cite[Lemma 5.1.4]{helein_conservation} with a new proof which works for all $m$. We also recover (when $m=2$) $\eps=\sqrt{\tfrac{4\pi}{3}}$ using the optimal $L^2$-Wente estimates proved by Ge \cite{G98}, see Remark \ref{rmk:m=2} for the necessary changes in the proof.  This constant is equivalent to the $\sqrt{\frac{8\pi}{3}}$ appearing in \cite{helein_conservation}, the discrepancy is due to a differnent choice of metric on $G(n,d)$. 
\end{remark}

\paragraph{Outline of the Proof of Theorem \ref{main}} The proof relies on a series of observations concerning the geometry of the connection $\D^\om$ which we expect to be of interest in their own right.  

Let $\End(\T\R^d)$ denote the endomorphism bundle restricted to $\cN$ and $\tilde{\T}\in \Gamma(\End(\T\R^d))$ be the projection onto the tangent space of $\cN$: $\tilde{\T}(z) = \Pi_{\T_z\cN}$. Similarly $\tilde{\V}\in\Gamma(\End(\T\R^d))$ projection onto the normal space $\tilde{\V}(z) = \Pi_{\V_z\cN}$. Denote by $\T=\tilde{\T} \circ u \in \Gamma(u^\ast\End(\T\R^d))$, $\V = \tilde{\V}\circ u \in \Gamma(u^\ast\End(\T\R^d))$ and $\cR=\T-\V\in \Gamma(u^\ast\End(\T\R^d))$ where we note that $\cR$ is of course orthogonal with $\cR=\cR^{-1}=\cR^T$. In fact we have the following identities for $\om$ (see Lemma \ref{lem:idom}):
$$\om = \fr12 \cR^{-1}\ed \cR = \T \ed \T - \ed \T \T.$$

 The connection $\D^\om$ on $u^\ast (\T \R^d)$ induces a connection on $u^\ast \End(\T\R^d)$ and we show that $\T$, $\V$ and $\cR$ are all parallel sections (see Lemma \ref{lem:parallel}), in particular: 
$$\D^\om \cR = \ed \cR + [\om ,\cR] = 0.$$
Since $\|\om\|_{M^{2,m-2}}$ is small, we utilise the Coulomb frame constructed by Rivi\`ere-Struwe \cite{RS08} and find a new frame $P$ making the new connection forms divergence free. At which point the gauge-invariant form of the above allows us to write, for $Q=P^{-1}\cR P$: 
$$\ed Q=-[\ed^\ast \xi,Q].$$
Testing this equation with $\ed Q$ gives a Wente structure on the right hand side and we may employ $\h^1$-$BMO$ duality to show that $\|\ed Q\|_{L^2(B_1)}^2\leq C\eps\|\ed Q\|_{L^2(B_1)}^2$ i.e. when $\eps$ is sufficently small $Q=P^{-1}\cR P$ is a constant. In particular this yields that $P$ itself simultaneously trivialises both the pulled-back tangent and normal bundles. The full details are given in Section \ref{sec:p1}.

\subsection{The regularity of harmonic maps into homogeneous targets}
As discussed in the opening weakly harmonic maps are regular in a neighbourhood of a point $p$ if $\|\ed u\|_{M^{2,m-2}(B_R(p))}$ is sufficiently small for some radius $R>0$. By the scale invariance of this norm, this condition may be stated in an equivalent way: as long as all approximate tangent maps, $\hat{u}_{r,q}(x) = u(q+rx)$, for $q\in B_R(p)$ and all $r<R-|p-q|$, are locally $W^{1,2}$-close to a constant, then $u$ is regular near $p$.

In the case of harmonic maps into homogeneous targets, Helein \cite{helein_div_free} utilised Noether's theorem to show that pure Wente structures naturally appear using the presence of a moving frame constructed via the isometries acting on $\cN$ (and not via Coulomb gauge methods). In the second part of this article we note that weakly harmonic maps into a homogeneous target are regular in a neighbourhood of a point $p$ if 
$$[u]_{BMO(B_R(p))} :=\sup_{B_r(x)\In B_R(p)} 
 \left( r^{-m}\int_{B_r(x)}  |u- \overline{u}_{r,x}|^2  \id x \right)^{1/2}, \qquad \overline{u}_{r,x}:=\fr{1}{|B_r(x)|}\int_{B_r(x)} u$$ is sufficiently small - i.e. when $\cN$ is homogeneous we may replace a first order condition on the map $u$ ($\|\ed u\|_{M^{2,m-2}}$ small) with a zeroth order one ($[u]_{BMO}$ small) to achieve full regularity. As above, we may re-state this: for weakly harmonic maps into homogeneous targets, if all approximate tangent maps in a neighbourhood of $p$ are locally $L^2$-close to a constant, then $u$ is regular near $p$. 

\begin{theorem}\label{thm:bmoreg}
Suppose that $\cN^n\emb \R^d$ is a closed homogeneous Riemannian manifold isometrically embedded in Euclidean space and equip $B_1^m$ with a smooth metric $g$. Then there exists $C=C(m,g,\cN)<\infty$ so that if $u:(B_1,g)\to\cN$ is weakly harmonic then  
$$\|\D^2 u\|_{L^1(B_{1/2})}\leq C\|\D u\|_{L^2}^2.$$ Furthermore there exists $\eps>0=\eps(m,g,\cN)>0$ so that for any weakly harmonic map $u:(B_1,g)\to\cN$ satisfying 
$[u]_{BMO(B_1)}<\eps$
then 	
 $$\|\D u\|_{L^\infty(B_{1/2})}\leq C\|u-\overline{u}\|_{L^1(B_1)}\leq C[u]_{BMO(B_1)}.$$
\end{theorem}

Our starting point for the proof is a result of H\'elein \cite[Lemmata 1 and 2]{helein_div_free} which uses the homogeneity of $\cN$ and Noether's theorem to uncover Wente terms in the harmonic map PDE, see also \cite{BR25} for a new interpretation of this via equivariant embeddings. From this we may infer that there exist $K=K(\cN)\in \N$, $L=L(\cN)<\infty$ so that a harmonic map $u$ as in the Theorem above solves 
\begin{equation}\label{eq:uhom}
	\Dl_g u = X_j \cdot_g \ed (Y^j(u))
\end{equation}
for  
\begin{equation}\label{eq:X}
\{X_j\}_{j=1}^K\In L^2(B_1,\T^\ast \R^m), \quad \ed^\ast_g(X_j)= 0 	\quad \text{and} \quad |X_j|\leq L|\ed u|,
\end{equation}   
and  
\begin{equation}\label{eq:Y}
\{Y^j\}_{j=1}^K\In \Gamma(\T\cN), \quad 	|\ed (Y^j(u))|\leq L|\ed u|,\quad \text{and}\quad [Y^j(u)]_{BMO(B_1)}\leq L[u]_{BMO(B_1)}.
\end{equation}

The proof of Theorem \ref{thm:bmoreg} now follows from the below, the proof of which is given in Section \ref{sec:p2}. We note that by translation in $\R^d$ we may assume that \eqref{eq:uhom}-\eqref{eq:Y} remain true for $u-\overline{u}$ in place of $u$, where $\overline{u}=\frac{1}{|B_1|}\int_{B_1} u = 0$:     

 \begin{theorem}\label{thm:bmoregsup}
    Let $1\leq m, n$ and $B_1^m \In \R^m$ the unit ball equipped with a smooth Riemannian metric $g$. Assume that $u \in W^{1,2}(B_1^m,\R^d)$ weakly solves 
    \begin{equation*} 
\Delta_g u = X_j \cdot_g \ed Y^j \quad \text{in } B_1^m
\end{equation*}
where $\{X_j\}$, $\{Y^j\}$ satisfy \eqref{eq:X} and \eqref{eq:Y} and we assume additionally that $\int_{B_1} u =0$.  Then there exists $C=C(K,L,m,g,d)$ so that $$\|\D^2 u\|_{L^1(B_{1/2})}\leq C\|\D u\|_{L^2}^2.$$ Furthermore there exits $\eps=\eps(K,L,m,g,d)>0$ so that $[u]_{BMO(B_1)}<\eps$ implies that 
$$\|\D u\|_{L^\infty(B_{1/2})}\leq C\|u\|_{L^1(B_1)}\leq C[u]_{BMO(B_1)}.$$
\end{theorem}

\paragraph{Acknowledgements}
The authors were supported by the EPSRC grant EP/W026597/1.

\section{Proof of Theorems \ref{main} and \ref{thm:gen}}\label{sec:p1}
Given a $C^2$-Riemannian manifold $\cN\emb \R^d$, recall that $\tilde\T : \cN \to \End (\T \R^d)$ is the projection onto $\T \cN$ whilst $\tilde\V = \Id - \tilde\T$ is the projection onto the normal bundle $\V \cN$. Given a map $u:B_1^m \to \cN$ we set $\T=\tilde\T \circ u$, $\V = \tilde\V \circ u$. 

\begin{remark}\label{rmk:T}
	We note trivially from the definition of $\T$ and $\V$ that $\T + \V \equiv \Id $ and so $\ed \T = - \ed \V$, but also  $\V\ed \T \V =0=\T \ed \T \T $ since:
	$$\T=\T^2 \implies \ed \T = \T \ed \T + \ed \T \T \implies \T \ed \T \T = 2\T \ed \T \T \quad\text{similarly} \quad \V\ed \V \V=0 \implies \V\ed \T \V =0.$$ 
In particular we may decompose $\ed \T$ into the sum of two parts, one which maps tangent vectors to normal, and its transpose: 
	$$\ed \T = \V \ed \T \T + \T \ed \T \V = \ed \T \T + \T \ed \T.$$

\end{remark}
\begin{lemma}\label{lem:idom}
	Let $\cN^n\emb \R^d$ be a $C^2$-Riemannian manifold and  $u\in W^{1,2}(B_1,\cN)$ with $\om$ defined by \eqref{eq:omdef}. Then we may write $\om$ in the following two equivalent ways: 
	\begin{enumerate}
		\item Setting $\cR=\T-\V\in W^{1,2}(B_1^m,\Or(d))$ we have $\om = \fr12 \cR^{-1}\ed \cR$. 
		\item $\om = \T\ed \T - \ed \T \T$. 
	\end{enumerate}
\end{lemma}

\begin{remark}\label{rmk:totgeo}
	 As may be seen in the work of Uhlenbeck \cite[Section 8]{uchi}, or indeed checked directly, $\mathfrak{O}=\{O\in \Or(d):O^2 = \Id\}$ is a disjoint union of totally geodesic submanifolds of $\Or(d)$ and that $\tilde{\cR}:G(n,d) \to \Or(d)$ given by $\tilde{\cR}(E)= \Pi_{E} - \Pi_{E^\bot}$ may be considered a totally geodesic embedding onto $\mathfrak{O}_n=\{O\in \mathfrak{O} : \textrm{dim}_{+1}(O)=n\}$ where $\textrm{dim}_{+1}(O)$ is the dimension of the eigenspace corresponding to eigenvalue $+1$. Thus $\cR$ above should be thought of as the composition of $u$ with the Gauss map $G:\cN\to G(n,d)=\mathfrak{O}_n$.  
\end{remark}

\begin{remark}\label{rmk:2om}
	Note that the first point in Lemma \ref{lem:idom} indicates immediately that the related connection $\D^{2\om} = d + 2\om$ is flat. We also note later that if we additionally assume that the tension field of $u$, $\tau(u):=\T \Dl u$ is weakly in $L^1$ then we may compute (see the Appendix), for any $v\in \R^d$: 
\begin{equation}\label{eq:divom}
	-\ed^\ast \om v = -\A(v,\tau(u)) + \langle \A(\tau(u), \cdot)^\sharp,v\rangle -\D^\cN_{\ed u} \A (v, \ed u) + \langle \D^\cN_{\ed u}\A (\ed u, \cdot)^\sharp, v\rangle.
\end{equation} Thus if $u$ is harmonic, and $\cN$ is embedded via a parallel second fundamental form ($\D^\cN \A \equiv 0$) then $\ed^\ast \om = 0$, too. Examples of such $\cN$ include round spheres, $\Or(d)$, $\mathrm{U}(d)$, and real and complex Grassmannians equipped with their induced metrics e.g. from the embedding $\tilde{\mathcal{R}}$ above.      \end{remark}

\begin{proof}

For $z_0\in \cN$ fix an orthonormal frame $ \left\{ e_i(z) \right\}_{i=1}^n $ around $ z_0 \in \cN $ such that $ \nabla^\cN_{e_i} e_j(z_0) = D_{e_i} e_j(z_0) - \A(e_i, e_j) = 0 $, where we denote by $ \nabla^\cN $ the Levi-Civita connection on $ \cN $ and by $ D $ the usual Euclidean derivative in $ \R^d $. Observe that for all $ v \in \R^d $, we have
\begin{equation*}
    \Tilde{\T}(z)v = \sum_{i=1}^n \langle e_i(z), v \rangle e_i(z),
\end{equation*}
thus
\begin{equation*}
       \ed \Tilde{\T}(z)\left[ e_j \right] v = \sum_{i=1}^d \langle D_{e_j} e_i(z), v \rangle e_i(z) + \langle e_i, v \rangle D_{e_j} e_i(z).
\end{equation*}
Evaluating this at $ z_0 $, we obtain
\begin{equation*}
       \ed \Tilde{\T}(z_0) \left[ e_j \right] v = \sum_{i=1}^n \langle \A(e_i, e_j), v \rangle e_i + \langle e_i, v \rangle \A(e_i, e_j).
\end{equation*}
Now, if $ X \in T_{z_0} \cN $, using Einstein's summation convention, we can write $ X = X^j e_j $, and the above formula can be expressed as
\begin{equation}\label{eq:DT}
       \ed \Tilde{\T}(z_0) \left[ X \right] v = \sum_{i=1}^n \langle \A(e_i, X), v\rangle e_i + \A(v, X) = \langle \A(\cdot, X)^\sharp, v \rangle + \A(v, X). 
\end{equation}
Thus
\begin{equation*}
    \ed \T(x)  v = \langle \A(\cdot, \ed u)^\sharp, v \rangle + \A(v, \ed u)
\end{equation*}
and clearly
\begin{equation*}
	\ed \T (\T v)= \A(v,\ed u) \qquad \text{and} \qquad \T \ed \T v= \langle \A(\cdot, \ed u)^\sharp, v \rangle
\end{equation*}
which, by comparison with \eqref{eq:omv} gives $\om v = \T \ed \T v - \ed \T \T v$ for all $v\in \R^d$, proving the second claim.

Now, since $ \T $ and $ \cV $ are projections, it is straightforward to verify that $ \cR^T = \cR $ and $ \cR^2 = \Id $, which implies $ \cR^{-1} = \cR^T = \cR $. As a consequence, we have using Remark \ref{rmk:T}
\begin{align*}
    \frac{1}{2} \cR^{-1} \ed \cR &= \frac{1}{2} \cR \ed \cR = \frac{1}{2} \cR \left( \ed \T - \ed \cV \right) = \cR \ed \T = \left( \T - \cV \right) \ed \T = \T \ed \T - \cV \ed \T \\
    &=\T \ed \T - \ed \T \T 
\end{align*}
since $\V \ed \T = \ed \T \T$. 

\end{proof}

The connection $\D^\om$ on $u^\ast \T\R^d$ induces also connection on $u^\ast (\End(\T \R^d)$ via 
$$\D^\om E = \ed E + [\om ,E] \qquad \text{for $E\in \Gamma(u^\ast\End(\T \R^d))$}$$
from which it is straightforward to check:

\begin{lemma}\label{lem:parallel}
$\D^\om \T =0$, $\D^\om \V = 0$ and $\D^\om \cR =0$. 
\end{lemma}
\begin{proof}
We will check only the first claim, with the others following trivially from this and Remark \ref{rmk:T}. Using part $2.$ of Lemma \ref{lem:idom} and Remark \ref{rmk:T} we have 
$$[\om ,T] = \om \T -\T \om = (\T \ed \T - \ed \T \T)\T - \T(\T \ed \T - \ed \T \T)=-\ed \T \T - \T \ed \T = -\ed \T $$
as required. 	
\end{proof}

\begin{proof}[Proof of Theorem \ref{main} and Theorem \ref{thm:gen}]

In the case of Theorem \ref{thm:gen} we note that Remark \ref{rmk:T} and Lemma \ref{lem:parallel} holds for $\Pi$, $\Pi^\bot$, and $\om_\Pi$ respectively, and indeed the below proof works by replacing $\T$ with $\Pi$, $\V$ with $\Pi^\bot$, $\om$ with $\om_\Pi$ and $\mathcal{R}$ with $\Pi-\Pi^\bot$. 

Using the Morrey-space version of the Coulomb frame proven by Rivi\`ere-Struwe \cite[Lemma 3.1]{RS08} we know that by choosing $\eps = \eps(m,d)>0$ sufficiently small there are \( P \in H^1(B_1^m, \SO(d)) \) and \( \xi \in H_0^1(B_1^m, \so(d) \otimes \wedge^{2} \T^\ast\mathbb{R}^m) \) such that
\begin{equation} \label{eq:10}
P^{-1} \ed P + P^{-1} \om P =  \ed^\ast \xi, \quad \text{and} \quad \ed \xi = 0 \quad \text{on } B_1^m.
\end{equation}
Moreover, \( \ed P \) and \( D \xi \) belong to \( M^{2,m-2}(B_1^m) \) with
\begin{equation}\label{eq:11}
\| \ed P \|_{M^{2,m-2}} + \| D \xi \|_{M^{2,m-2}} \leq C \| \om \|_{M^{2,m-2}} \leq C \varepsilon .
\end{equation}
In the above $D\xi$ denotes the collection of all first order derivatives of all components of $\xi$ which we distinguish from $\ed \xi$ which is simply the exterior derivative of the two-form.

The gauge-invariant version of Lemma \ref{lem:parallel} tells us that $$\D^{\ed^\ast \xi}(P^{-1}\T P) = \ed (P^{-1}\T P)+ [\ed^\ast \xi, P^{-1}\T P] = 0,$$ similarly  $\D^{\ed^\ast \xi}(P^{-1}\V P) = 0$ and $\D^{\ed^\ast \xi}(P^{-1}\cR P) = 0$, as may be checked directly using \eqref{eq:10}.

In particular letting $Q=P^{-1}\cR P$ then we have $\ed Q=[Q,\ed^\ast \xi]$ giving (where we always sum over repeated indices)
\begin{eqnarray*}
|\ed Q|^2 &=& \textrm{tr}( [Q,\ed^\ast \xi]\cdot \ed Q^T) =  \ast(Q^i_k\ed^\ast \xi^k_l \wedge \ast \ed Q^i_l- \ed^\ast \xi^k_l Q^l_i \wedge \ast \ed Q^k_i) \\
&=&-\ast 2(\ed \ast \xi^k_l \wedge Q^i_k\ed Q^i_l ) \\
&=& -\ast 2\ed (\ast \xi^k_l \wedge Q^i_k\ed Q^i_l)+ 2(-1)^{m-2}\ast( \ast \xi^k_l \wedge (\ed Q^i_k \wedge \ed Q^i_l) ).
\end{eqnarray*}

We may extend $\xi$ by zero to the whole of $\R^m$ without relabelling, and recall that $[\xi]_{BMO(\R^m)}\leq C\|D \xi\|_{M^{2,m-2}(\R^m)}$ in view of the Poincar\'e inequality.  We also extend $Q-\overline{Q}$ to a $W^{1,2}$ function $
\tilde{Q}$ on the whole of $\R^m$ so that $\|\ed \tilde{Q}\|_{L^2}\leq C\|\ed Q\|_{L^2(B_1)}$ and $\ed \tilde{Q} = \ed Q$ in $B_1$. By $\h^1$-BMO duality \cite{FSt72} and the fundamental results of Coiffman-Lions-Mayers-Semmes \cite{CLMS93} we have, up to a sign which is irrelevnat  
\begin{eqnarray*}
	\|\ed Q\|_{L^2(B_1)}^2 &=& 2 \int
 \ast \xi^k_l \wedge (\ed \ti{Q}^i_k \wedge \ed \ti{Q}^i_l) \\
	&\leq & 2[\xi^k_l]_{BMO} \|\ed \ti{Q}^i_k \wedge \ed \ti{Q}^i_l\|_{\h^1} \\
	&\leq& C \|D \xi\|_{M^{2,m-2}}\|\ed Q\|_{L^2(B_1)}^2\\ 
	&\leq &  C\eps \|\ed Q\|_{L^2(B_1)}^2.
\end{eqnarray*}
For $\eps$ sufficiently small, $Q$ is a constant. Now if $\{E_i\}$ and $\{N_j\}$ is an orthonormal basis $\T_{u(0)}\cN$ and $\V_{u(0)}\cN$ then we

\textbf{Claim:} $e_i(x) := P(x)P(0)^T E_i$ and $\nu_j(x):= P(x)P(0)^T N_j$ are orthonormal moving frames for $u^\ast(\T \cN)$ and $u^\ast(\V \cN)$ respectively as required by the Theorem.

Clearly both $\{e_i\}$ and $\{\nu_j\}$ are orthonormal frames satisfying the desired estimate, by \eqref{eq:11}. We will see below that $\cR(x)e_i(x) = (\T(x)-\V(x))e_i(x) = e_i(x)$ meaning that $e_i(x) \in \T_{u(x)}\cN$ almost everywhere, and similarly $\cR(x)\nu_j(x) = - \nu_j(x)$ meaning that $\nu_j(x)\in \V_{u(x)}\cN$ almost everywhere. Indeed:  
\begin{eqnarray*}
	\cR(x) e_i(x)&=&\cR(x)P(x)P(0)^T E_i=P(x) (P^{T}(x)\cR (x) P(x)) (P(0)^T E_i) \\
	&=&P(x)(P(0)^T \cR(0) P(0))P(0)^T E_i = e_i(x)
\end{eqnarray*}
and 
\begin{eqnarray*}
	\cR(x) \nu_j(x)&=&\cR(x)P(x)P(0)^T N_j=P(x) (P^{T}(x)\cR (x) P(x)) (P(0)^T N_j) \\
	&=&P(x)(P(0)^T \cR(0) P(0))P(0)^T N_j = -\nu_j(x). 
\end{eqnarray*}
It remains to check that $\ed^\ast (e_i\cdot \ed e_j)=0$ and similarly for $\{\nu_j\}$. We do this for the former leaving the latter to the reader. Assume w.l.o.g. that $P(0)=\Id$ and note that 
$$e_i\cdot \ed e_j = PE_i\cdot (\ed P)E_j = E_i \cdot P^{-1}\ed P E_j =E_i\cdot \ed ^\ast \xi E_j - E_i \cdot (P^{-1}\om P) E_j.$$
The result follows from $E_i \cdot (P^{-1}\om P) E_j=e_i \cdot \om e_j=0$ since $\om e_j$ is a normal vector by Remark \ref{rmk:T}.	
\end{proof}

\begin{remark}\label{rmk:m=2}
When $m=2$ one finds \( P \in H^1(B_1^m, \SO(d)) \) and \( \xi \in H_0^1(B_1^m, \so(d))\) solving
\begin{equation*}
P^{-1} \ed P + P^{-1} \om P =  \ast \ed \xi \quad \text{on } B, \quad \text{with}\quad \| \ed P \|_{L^2}\leq 2\|\om\|_{L^2} \quad\text{and} \quad \| \ed \xi \|_{L^2} \leq \| \om \|_{L^2}\end{equation*}
regardless of the size of $\eps$ (see e.g. \cite[Theorem 8.4]{S14}). The optimal $L^2$ Wente estimate on dics \cite{G98} then gives 
$\|\ed Q\|_{L^2}^2\leq \sqrt{\frac{3}{4\pi}}\|\om\|_{L^2}\|\ed Q\|_{L^2}^2$ which means that $\eps = \sqrt{\frac{4\pi}{3}}$ is sufficient. 
\end{remark}
\section{Proof of Theorem \ref{thm:bmoreg}}\label{sec:p2}

\begin{proof}[Proof of Theorem \ref{thm:bmoregsup}]
   For simplicity we will give the proof only in the case that $g$ is the Euclidean metric, leaving the required  changes necessary for the general case to the interested reader. 
   
   The first estimate on $\|\D^2 u\|_{L^1(B_{1/2})}$ is standard - by suitably extending $X_j$ and $Y^j$ similarly to the below one may use e.g. \cite[Theorem 3.2.9]{helein_conservation}.
      
  For the $BMO$-regularity part we begin by noticing that, from standard $L^2$-Hodge theory (see for instance \cite[Corollary 10.5.1]{IM01}), there exists $ \xi_j \in W^{1,2}(B_1^m,  \wedge^{2} \T^\ast \mathbb{R}^m) $ such that $X_j = \ed^\ast \xi_j$, each component of $\xi_j$ has mean zero in $B_1$ and $\|D{\xi}_j\|_{L^2(\R^m)}\leq \|X_j\|_{L^2(B_1)}$. We extend $\xi_j$ to $\tilde{\xi}_j\in W^{1,2}(\R^m,\wedge^2\T^\ast\R^m)$ and for which we still have
    \[
    \Vert D \tilde{\xi}_j \Vert_{L^2} \leq C \Vert X_j \Vert_{L^2(B_1)}\leq C\|\ed u\|_{L^2(B_1)}.
    \]
We similarly extend each $Y^j-\overline{Y}^j$ to $\tilde{Y}^j$, defined on the whole of $\R^m$ and for which we have $\ed \tilde{Y}^j = \ed Y^j$ in $B_1$ and (see \cite{Jo80}) 
    $$[\tilde{Y}^j]_{BMO(\R^m)}\leq C[Y^j]_{BMO(B_1)}\leq C[u]_{BMO(B_1)}.$$

   Testing the equation with $\phi \in C_c^\infty(B_1^m, \R^d)$ and applying the divergence theorem, we obtain
    \begin{eqnarray*}
    \int   X_j \cdot \ed Y^j \phi &=& \int \ed^\ast \tilde{\xi}_j \cdot \ed \tilde{Y}^j \phi = -\int_{B_1} (  \ed^\ast \tilde{\xi}_j \cdot \ed \phi ) \tilde{Y}^j \\
    &=& \int ( \ed \ast \tilde{\xi}_j \wedge \ed \phi)\tilde{Y}^j.
     \end{eqnarray*}
 Now, \cite{CLMS93} gives that 
 $$\|( \ed \ast \tilde{\xi}_j \wedge \ed \phi)\|_{\h^1(\R^m)}\leq C\| D \tilde\xi_j\|_{L^2(\R^m)}\|\ed \phi\|_{L^2(B_1)}\leq C\|\ed u\|_{L^2(B_1)}\|\ed \phi\|_{L^2(B_1)}.$$ 
 Combined with the $\h^1-BMO$  duality we have, for all $\phi \in C_c^\infty(B_1^m, \R^d)$ 
\begin{eqnarray*}
  \int   X_j \cdot \ed Y^j \phi   &\leq& C\|\ed u\|_{L^2(B_1)}\|\ed \phi\|_{L^2(B_1)} [\tilde{Y}^j]_{BMO(\R^m)}  \\
    &\leq & C\|\ed u\|_{L^2(B_1)}[u]_{BMO(B_1)} \|\ed \phi\|_{L^2(B_1)}, \notag
    \end{eqnarray*}
    giving $\|X_j\cdot \ed Y^j\|_{H^{-1}(B_1)}\leq C\|\ed u\|_{L^2(B_1)}[u]_{BMO(B_1)}$. Now letting $v\in W_0^{1,2}$ be the unique solution to 
    $\Dl v^i =  X_j \cdot \ed Y^j$ 
    we have 
    \begin{equation}\label{eq:dv} 
    \|v\|_{L^2(B_1)}^2+\|\ed v\|_{L^2(B_1)}^2 \leq C\|\ed u\|_{L^2(B_1)}^2[u]_{BMO(B_1)}^2.
    \end{equation}
    
    We now may essentially follow some of the arguments as in \cite{ST13} to complete the proof which we outline briefly: 
    
    Write $u=h+v$ for $h$ a harmonic function. We first observe that $\|h\|_{L^1(B_1)}\leq C(\|u\|_{L^1(B_1)} + \|v\|_{L^1(B_1)})\leq C(\|u\|_{L^1(B_1)}+\|\ed u\|_{L^2(B_1)}[u]_{BMO(B_1)})$, whence one trivially has that 
    $$\|\ed h\|^2_{L^2(B_{1/2})} \leq C(\|u\|_{L^1(B_1)}^2+\|\ed u\|_{L^2(B_1)}^2[u]_{BMO(B_1)}^2).$$ 
    
    Thus in particular, 
 $$\|\ed u\|^2_{L^2(B_{1/2})} \leq C(\|\ed u\|_{L^2(B_1)}^2[u]_{BMO(B_1)}^2 + \|u\|_{L^1(B_1)}^2)$$
and by repeating the argument above on each ball $B_{2R}(x)\In B_1$, and by the scaling properties of each quantity we see that 
\begin{eqnarray*}
	\|\ed u\|^2_{L^2(B_{R/2(x)})} &\leq& C(\|\ed u\|_{L^2(B_R(x))}^2[u]_{BMO(B_1)}^2 + R^{-m-2}\|u\|_{L^1(B_1)}^2)\\
	&\leq & C\eps^2\|\ed u\|_{L^2(B_R(x))}^2 + CR^{-m-2} \|u\|_{L^1(B_1)}^2. 
\end{eqnarray*}
Hence exactly as equation (23) in \cite{ST13} is derived using \cite[Lemma A.7]{ST13} we may obtain that, for $\eps$ sufficiently small
\begin{equation}\label{eq:dul1}
	\|\ed u\|_{L^2(B_{7/8})}\leq C\|u\|_{L^1(B_1)}. 
\end{equation}

 Observe now that we can write $\ed u = H + \ed v$ where $H=\ed h$ is a harmonic one-form and 
    \begin{equation*} 
    \|\ed u\|_{L^2(B_1)}^2 = \|H\|_{L^2(B_1)}^2 + \|\ed v\|_{L^2(B_1)}^2.
    \end{equation*}
    Since $H$ is harmonic, the quantity $r^{-m}\|H\|_{L^2(B_r)}^2 $ is increasing, which implies
    \begin{eqnarray} \label{eq:H}
    \|H\|_{L^2(B_r)}^2 &\leq& r^m \|H\|_{L^2(B_1)}^2 \notag \\
    &\leq& r^m\|\ed u\|_{L^2(B_1)}^2.
    \end{eqnarray}
    for all $ 0 < r \leq 1$. Using equations \eqref{eq:dv} and \eqref{eq:H}, for any $\dl >0$, we have
    \begin{eqnarray} \label{eq:decay}
    \|\ed u\|_{L^2(B_r)}^2  &\leq& (\|H\|_{L^2(B_r)} + \|\ed v\|_{L^2(B_r)} )^2 \notag \\
    &\leq& (1+\dl)\|H\|_{L^2(B_r)}^2 + C_{\dl} \|\ed v\|_{L^2(B_1)}^2 \notag \\
    &\leq& (1+\dl)r^m\|\ed u\|_{L^2(B_1)}^2 + C_{\dl}\|\ed u\|_{L^2}^2[u]_{BMO}^2\nn \\
    &=&  \left((1+\dl)r^m +C_{\dl}[u]_{BMO}^2\right) \|\ed u\|_{L^2(B_1)}^2 \notag\\
    &<& \left((1+\dl)r^m +C_{\dl}\eps^2\right)\|\ed u\|_{L^2(B_1)}^2
    \end{eqnarray}

    where we also used Young's inequality. At this point, given $\alpha \in (0,2)$ to be determined later, choose $\delta, \varepsilon \in (1,2]$ sufficiently small so that 
    \begin{equation*}
    \frac{(1+2\delta)}{2^m}+C_\delta \varepsilon^2=2^{\alpha-m}.
    \end{equation*}
    For any $x\in B_{3/4}$, $\varrho = 7/8 - |x|>1/8$, and any $r\leq \rho$ we may write $2^{-k-1} \varrho \leq r \leq 2^{-k} \varrho$. Deriving \eqref{eq:decay} on $B_\varrho(x)$ and iterating gives (also from \eqref{eq:dul1})
    \begin{equation*}
        \Vert \ed u \Vert^2_{L^2(B_r(x))} \leq \left( 2^{-k+1} \right)^{m-\alpha} \Vert \ed u \Vert^2_{L^2(B_\varrho(x))} \leq 32^{m-\alpha} r^{m-\alpha} \Vert \ed u \Vert^2_{L^2(B_{7/8})}\leq Cr^{m-\al}\|u\|_{L^1(B_1)}^2.
    \end{equation*}
    The above estimate gives
    \begin{equation*}
        \Vert \ed u \Vert_{M^{2.m-\alpha}(B_{3/4})} \leq C \|u\|_{L^1(B_1)}.
    \end{equation*}
    As a consequence, since $u$ solves equation \eqref{eq:uhom} and recalling \eqref{eq:X} and \eqref{eq:Y}, we obtain that 
    \[
    \Vert \Delta u\Vert_{M^{1,m-\alpha}(B_{3/4})} \leq C \|u\|_{L^1(B_1)}^2
    \]
    and thus, fixing $\al$ such that $\alpha$ such that $1< \alpha < 2m /(2m-1)$, the Riesz potential estimates of Adams \cite{Ad75} gives $\ed u\in L^{\tilde{p}}(B_{5/8})$ for
    $
    \tfrac{1}{\tilde{p}}=1-\tfrac{1}{\alpha}.
    $  Thus for some $q>2m$, 
    $$\|\ed u\|_{L^{q}(B_{5/8})}\leq C\|u\|_{L^1(B_1)}.$$
   Going back to \eqref{eq:X} and \eqref{eq:Y} we obtain $\|\Dl u\|_{L^s(B_{5/8})} \leq C\|u\|_{L^1(B_1)}^2$ for some $s>m$ from which Calderon-Zygmund estimates and Sobolev embedding finishes the proof.  
\end{proof}

\appendix
\section{Sketch proof of \eqref{eq:divom}}

\begin{proof}
Given $X,Y\in \Gamma(\T\cN)$ and assuming that $\D^\cN_X Y = \D^\cN_Y X=0$ at $z_0$ we have 
$$D^2\tilde{\T}[X,Y]v =D_Y (D_X \tilde{\T})v$$ and starting from  \eqref{eq:DT} we have, for any $v\in \R^d$
\begin{eqnarray*}
    D^2\tilde{\T}[X,Y]v &=& D_Y(\langle \A(\cdot, X)^\sharp, v \rangle + \A(v, X)) \\
    &=& \langle D_Y(\A(\cdot, X)^\sharp), v\rangle + D_Y(\A(v, X)) \\
    &=& \langle \D_Y^\cN \A (\cdot, X)^\sharp, v\rangle  -\langle \A(\cdot, X)^\sharp , \A(Y,v) \rangle + \D_Y^\cN \A(v,X) - \langle A(v,X), A(Y,\cdot)^\sharp\rangle 
\end{eqnarray*}
where we have used that if $\eta\in \Gamma(\V \cN)$ then $D_Y \eta = \D^{\cN}_Y \eta - \langle \eta, \A(Y,\cdot)^\sharp\rangle$. 

Thus we have that 
\begin{eqnarray*}
   \ed^\ast \om v &=& \Dl \T \T v - \T \Dl \T v \\
   &=& \ed \tilde{\T}[\T \Dl u] \T v - \T \ed \tilde{\T}[\T\Dl u] v + D^2\tilde{\T}[\ed u, \ed u] \T v - \T D^2\tilde{\T}[\ed u,\ed u] v \\
   &=& \A(v,\tau(u)) - \langle \A(\tau(u), \cdot)^\sharp,v\rangle +\D^\cN_{\ed u} \A (v, \ed u) - \langle \D^\cN_{\ed u}\A (\ed u, \cdot)^\sharp, v\rangle
\end{eqnarray*}
using the above, and \eqref{eq:DT} again.

\end{proof}

\bibliographystyle{amsplain}

\begin{flushleft}

L. Appolloni: \textit{l.appolloni@leeds.ac.uk} 

B. Sharp: \textit{b.g.sharp@leeds.ac.uk}  

School of Mathematics, University of Leeds, Leeds LS2 9JT, UK  
\end{flushleft}
\end{document}